\documentclass[12pt]{article}
\usepackage{amsmath,amssymb,amsfonts,amsthm,dsfont}
\usepackage{graphicx}
\usepackage{hyperref}
\usepackage{comment}
\usepackage{color}
\usepackage{tabularx}
\usepackage{caption}
\usepackage{enumerate}
\usepackage[normalem]{ulem}
\usepackage[usenames,dvipsnames]{xcolor}
\input epsf.tex
\usepackage{subfigure}

\hypersetup{
    linktoc=page,
   linkcolor=red,          
    citecolor=blue,        
    filecolor=blue,      
    urlcolor=cyan,
    colorlinks=true           
    }

\newtheorem{theorem}{Theorem}
\numberwithin{theorem}{section}

\newtheorem{proposition}[theorem]{Proposition}

\theoremstyle{remark}\newtheorem{remark}[theorem]{Remark}

\newcommand{\R}{\mathbb{R}}
\newcommand{\C}{\mathbb{C}}

\newcommand{\Z}{\mathbb{Z}}

\def\Ss{\mathcal{S}}

\def\supp{\mathrm{SUPP}}
\def\harm{\mathrm{HARM}}
\def\T{\mathcal{T}}

\def\SLEkk#1/{$\mathrm{SLE}_{#1}$}
\def\SLEk/{\SLEkk{\kappa}/}
\def\SLEtwo/{\SLEkk2/}
\def\SLE/{$\mathrm{SLE}$}
\def\CLEkk#1/{$\mathrm{CLE}_{#1}$}
\def\CLEk/{\CLEkk{\kappa}/}
\def\CLEtwo/{\CLEkk2/}
\def\CLE/{$\mathrm{CLE}$}
\def\GLEkk#1/{$\mathrm{GLE}_{#1}$}
\def\GLEk/{\GLEkk{\kappa}/}
\def\GLEtwo/{\GLEkk2/}
\def\GLE/{$\mathrm{GLE}$}

\def\Ito/{It\^o}

\def \E {{\bf E}}

\def\Var{\mathrm{Var}}
\def\Cov{\mathrm{Cov}}

\title{Log-correlated Gaussian fields: an overview}
\author{{\sc Bertrand Duplantier\thanks{e-mail: \texttt{bertrand.duplantier@cea.fr}.}},\, {\sc R\'{e}mi Rhodes\thanks{Partially supported by grant ANR-11-JCJC.}},\,
{\sc Scott Sheffield\thanks{e-mail: \texttt{sheffield@math.mit.edu}.
Partially supported by NSF grants DMS 064558 and DMS 1209044.}}, \\ and {\sc Vincent Vargas\thanks{e-mail: \texttt{vincent.vargas@ens.fr}. Partially supported by grant ANR-11-JCJC.}}
\\
{\it \normalsize $^*$Institut de Physique Th\'{e}orique, CEA/Saclay}\\
\vspace{.05in}
{\it \normalsize F-91191 Gif-sur-Yvette Cedex, France}\\
{\it  \normalsize $^\dagger$  Universit\'e Paris Est-Marne la Vall\'ee, LAMA, CNRS UMR 8050}\\
{\it \normalsize Cit\'e Descartes - 5 boulevard Descartes}\\
\vspace{.05in}
{\it \normalsize 77454 Marne-la-Vall�e Cedex 2, France}\\
{\it \normalsize $^\ddagger$Department of Mathematics}\\
{\it \normalsize Massachusetts Institute for Technology}\\
\vspace{.05in}
{\it \normalsize Cambridge, Massachusetts 02139, USA}\\
{\it \normalsize   $^\S$ Ecole Normale Sup\'erieure, DMA}\\
{\it \normalsize 45 rue d'Ulm}\\
{\it \normalsize 75005 Paris, France}
 }

\begin{document}
\maketitle
\vspace{-.3in}

\begin{abstract}
 We survey the properties of the log-correlated Gaussian field (LGF), which is a centered Gaussian random distribution (generalized function) $h$ on $\R^d$, defined up to a global additive constant.
Its law is determined by the covariance formula \begin{equation*} \Cov\bigl[ (h, \phi_1), (h, \phi_2) \bigr] = \int_{\R^d \times \R^d} -\log|y-z|  \phi_1(y) \phi_2(z)dydz,\end{equation*} which holds for mean-zero test functions $\phi_1, \phi_2$.  The LGF belongs to the larger family of {\em fractional Gaussian fields} obtained by applying fractional powers of the Laplacian to a white noise $W$ on $\R^d$.  It takes the form $h = (-\Delta)^{-d/4} W$.  By comparison, the Gaussian free field (GFF) takes the form $(-\Delta)^{-1/2} W$ in any dimension.  The LGFs with $d \in \{2,1\}$ coincide with the 2D GFF and its restriction to a line. These objects arise in the study of conformal field theory and SLE, random surfaces, random matrices, Liouville quantum gravity, and (when $d=1$) finance.  Higher dimensional LGFs appear in models of turbulence and early-universe cosmology.  LGFs are closely related to cascade models and Gaussian branching random walks.  We review LGF approximation schemes, restriction properties, Markov properties, conformal symmetries, and multiplicative chaos applications.
\end{abstract}

\tableofcontents

\section{Introduction}
\subsection{Overview}
The log-correlated Gaussian free field (LGF) is a beautiful and canonical Gaussian random generalized function (a {\em tempered distribution}) that can be defined (modulo a global additive constant) on $\R^d$ for any $d \geq 1$.We present a series of equivalent definitions of the LGF in Section \ref{ss:lcgfdefs}, but the simplest is that it is a centered real-valued Gaussian random tempered distribution $h$ on $\R^d$, defined modulo a global additive constant, whose law is determined by the covariance formula \begin{equation}\label{eqn::lgfcovariance}\Cov\bigl[ (h, \phi_1), (h, \phi_2) \bigr] = \int_{\R^d \times \R^d} -\log|y-z|  \phi_1(y) \phi_2(z)dydz,\end{equation} which holds for mean-zero test functions $\phi_1, \phi_2$.  Here ``mean-zero'' means $\int \phi_i(z) dz = 0$, which ensures that $(h,\phi_i)$ is well-defined even though $h$ is only defined modulo a global additive constant.  The statement that $h$ is {\em centered} means that $\mathbb E[(h,\phi)] = 0$ for each mean-zero test function $\phi$.

This note provides a brief overview of the LGF and is somewhat analogous to the survey of the Gaussian free field (GFF) presented in \cite{MR2322706}.  It should be accessible to any reader familiar with a few standard notions from real analysis (such as generalized functions, Green's function, and Gaussian Hilbert spaces).  The reader who is unfamiliar with Gaussian Hilbert spaces and Wiener chaos (as used in the construction of Brownian motion or the Gaussian free field) might wish to consult \cite{janson1997gaussian} or \cite{MR2322706} for a discussion of these issues with more detail than we include here. Also, as we will explain in Section \ref{sec::otherfgf}, the LGF belongs to the more general family of {\em fractional Gaussian fields} (FGFs) surveyed in \cite{fgfsurvey} (which is co-authored by one of the current authors).  This article can be viewed as a companion piece to \cite{fgfsurvey}.  We will cite \cite{fgfsurvey} for results that hold for general FGFs and emphasize here the results and perspectives that are specific to (or particularly natural for) the LGF.

The LGF in dimension $d=1$ has been proposed as a model of (the log of) financial market volatility \cite{BaMu, DRV}.
When $d =2$, the LGF coincides with the 2D Gaussian free field (GFF), which has an enormous range of applications to mathematical physics \cite{MR2322706}.   When $d=3$, the LGF plays an important role in early universe cosmology, where it approximately describes the gravitational potential function of the universe at a fixed time shortly after the big bang.\footnote{The Laplacian of the LGF is a random distribution $\psi$ that is believed to approximately describe perturbations of mass/energy density from uniformity.  This field follows the {\em Harrison-Zel'dovich} spectrum, which means that
$$\mathbb E[\hat \psi(k)^* \hat \psi (k')] \, =\,  \delta(k-k')P(k),$$
where $\delta$ is the Dirac delta function, the {\em power spectrum} $P(k)$ is given by $|k|^{n_s}$, and the {\em spectral index} $n_s$ is $1$.  (In the notation of Section \ref{sec::otherfgf}, $\psi$ is an FGF$_{-n_s/2}$.)  An overview of this story appears in the reference text \cite{2003moco.book.....D}.
This formula for $P$ has been empirically observed to {\em approximately} hold for a range of $k$ values spanning many orders of magnitude.  However, some of the most recent experimental data, including data from the Planck Observatory \cite{2013arXiv1303.5082P}, suggest that while the assumptions of Gaussianity, translation invariance, and rotational invariance are consistent with the data (and any ``non-Gaussianianity'' that exists must be limited), there are statistically significant differences between the empirically calculated power spectrum and the Harrison-Zel'dovich spectrum.  Section 2 of \cite{2013arXiv1303.5082P} provides a historical overview of this issue and many additional references, and Section 4 explains the recent observations.  Another analysis combining this data with more recent BICEP2 data finds $n_s$ between $.95$ and $.98$ \cite{2014arXiv1403.6462W}.}  When $d=4$ the LGF is the continuum analog of the so-called {\em Gaussian membrane model},\footnote{To be precise, if we are given a finite $\Lambda \subset \mathbb Z^d$ and a boundary function $h_\delta : (\mathbb Z^d \setminus \Lambda) \to \R$, then the membrane model on $\Lambda$ is a probability measure on the finite-dimensional space of functions from $\mathbb Z^d$ to $\R$ that agree with $h_\delta$ outside of $\Lambda$.  The probability density function is $e^{-H(h)/2}$ (times a normalizing constant), where $$H(h) = \sum_{v \in \overline{\Lambda}} (\Delta h(v) )^2$$ and $\Delta$ is the discrete Laplacian and $\overline \Lambda$ is the union of $\Lambda$ and set of vertices adjacent to a point in $\Lambda$.  The membrane model on all of $\Z^d$ can be defined as a Gibbs measure on functions $h$ from $\Z^d$ to $\R$ in the usual way (which, depending on $d$, may be defined modulo global additive constants or modulo discrete harmonic polynomials of some degree): conditioned on the $h$ values outside of $\Lambda$, the conditional law within $\Lambda$ is as described above.  Note that this conditional law depends on and is determined by the given values on those vertices of $\Z^d \setminus \Lambda$ whose graph distance from $\Lambda$ is $1$ or $2$.  By contrast, in the case of the discrete Gaussian free field, the conditional law would be determined by the values on the boundary vertices, i.e., those vertices in $\Z^d \setminus \Lambda$ whose graph distance from $\Lambda$ is $1$.} which is a Gibbs measure whose defining energy is the $L^2$ norm of the discrete Laplacian (c.f.\ the discrete Gaussian free field, whose defining energy function is the $L^2$ norm of the discrete gradient)\cite{kurt2009maximum, sakagawa2012free, cipriani2013high}.  The literature on these subjects is large and complex, and we will not attempt to survey it further here.

We will see that the LGF has conformal invariance symmetries in any dimension.  It is also closely related, in any dimension, to additive cascade models and branching random walks.
We will see that when $d$ is even the LGF has an interesting type of Markov property: namely, the conditional law of $h$ in a spherical domain $D$, {\em given} the behavior of $h$ outside of $D$, depends only on the given values of $h$ on $\partial D$ and the first $d/2-1$ normal derivatives of $h$ on $\partial D$.  We will also show for general $d$ that the restriction of an LGF on $\R^d$ to a lower dimensional subspace is an LGF on that subspace.

As mentioned above, in two dimensions, the GFF and the LGF coincide.  Recent years have seen an explosion of interesting results about the 2D GFF, and it is interesting to note that while many of these results can be naturally generalized to other dimensions, the generalizations often apply to general LGFs, not general GFFs.  This is true, for example, of several recent results about multiplicative chaos and the so-called KPZ formula \cite{MR2819163, rhodes-2008,2012arXiv1206.1671D,MR3215583,2012arXiv1210.8051C}.  In some cases, the generalizations are unsolved problems.   For example, it was shown in \cite{ContourLine} that even though the 2D GFF is a distribution, not a continuous function, it is possible to define continuous zero-height ``contour lines'' of the GFF using so-called SLE$_4$ curves (i.e., Schramm-Loewner evolutions with parameter $\kappa = 4$).  It remains an open question whether the level surfaces of the three (or higher) dimensional LGF can be canonically defined in a similar way (although it is possible to draw level surfaces of continuous {\em approximations} to the LGF, as Figure \ref{fig::levelsurfaces} illustrates).

\begin{figure}[h]
\begin{center}
\includegraphics[angle=0,width=.93290\linewidth]{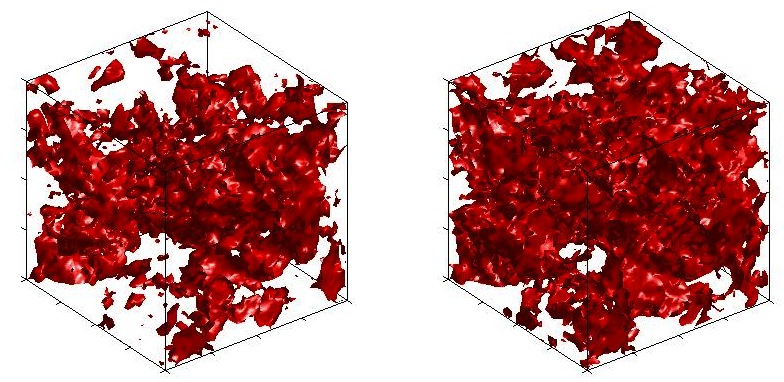}
\caption{\label{fig::levelsurfaces} Level surfaces of simulated fine-mesh approximations of the 3D LGF.  Since the restriction of a 3D LGF to the plane is a 2D GFF, we expect (based on the GFF results in \cite{ContourLine}) that the intersection of one of these surfaces with a plane will be comprised of loops that look locally like SLE$_4$.}
\label{fig1}
\end{center}
\end{figure}

\subsection{Relation to other fractional Gaussian fields} \label{sec::otherfgf}

The LGF belongs to a larger one-parameter family of Gaussian random functions (or generalized functions) on $\R^d$: namely, the family of Gaussian fields obtained as \begin{equation} \label{eqn::laplaceinverse} h = (-\Delta)^{-s/2} W\end{equation} where $s \in \R$, and $\Delta$ is the Laplacian, and $W$ is a white noise on $\R^d$.  Following the survey article \cite{fgfsurvey}, we will refer to a field of the form $(-\Delta)^{-s/2} W$ as a {\bf fractional Gaussian field} and denote it by FGF$_s$.  (The parameter $s$ is related to the so-called {\em Hurst parameter}, as we will explain below.)

Making sense of \eqref{eqn::laplaceinverse} requires us to make sense of the operator $(-\Delta)^{s/2}$ in this context.  Fractional powers of the Laplacian can be rigorously constructed in various ways and are the subject of a large and rapidly growing literature (see \cite{caffarelli2007extension, caffarelli2008regularity, ros2012dirichlet, fractionalwiki} and the references therein, as well as our discussion in Section \ref{subsec::basicdefs}) but they are not generally well-defined for {\em all} tempered distributions, so there is still some thought required to make sense of the $h$ in \eqref{eqn::laplaceinverse}.  As we explain below the basic idea is to define $(-\Delta)^{s/2}$ to be the operator that multiplies the Fourier transform of its input by the function $|\cdot|^s$.

It turns out that the LGF, as defined by \eqref{eqn::lgfcovariance}, is equivalent (up to a multiplicative constant) to \eqref{eqn::laplaceinverse} when $s = d/2$.  We will explain how to make sense of \eqref{eqn::laplaceinverse} when $s=d/2$ in Section \ref{ss:lcgfdefs}.
The survey \cite{fgfsurvey} contains more detail about general FGF$_s$ processes, and much of that analysis applies to the LGF as a special case.  For example, \cite{fgfsurvey} explicitly derives the constant relating \eqref{eqn::lgfcovariance} and \eqref{eqn::laplaceinverse}, which we will not do here.

The family of fractional Gaussian fields includes white noise itself ($s=0$) and the Gaussian free field ($s=1$).  The field $(-\Delta)^{-s/2} W$ can also be understood as a standard Gaussian in the Hilbert space with inner product \begin{equation} \label{eqn::hspaceinnerproductdef} \bigl( (-\Delta)^{s/2}f , (-\Delta)^{s/2} g\bigr) = \bigl((-\Delta)^s f, g\bigr),\end{equation}
as we will explain (in the LGF context) in Section \ref{ss:lcgfdefs}.

The {\bf Hurst parameter} of an FGF$_s$ is the quantity $$H  = s-d/2.$$  The Hurst parameter describes the scaling relation that an instance $h$ of the FGF$_s$ possesses: if $a>0$ is a constant, then $h(a z)$ has the same law as $|a|^H h(z)$.  Note that for the LGF we have $H = 0$ and  $s = d/2$.  The fact that $H=0$ is related to the fact that, as a random generalized function, the LGF is exactly invariant under conformal transformations of the domain (without rescaling the range) in any dimension, which is not true of FGF$_s$ when $s \not = d/2$.\footnote{If $\Phi$ is a conformal automorphism of $\R^d$, then we define the pullback $h \circ
\Phi^{-1}$ of $h$ to be the distribution defined by $(h
\circ \Phi^{-1}, \tilde \phi) = (h, \phi)$ whenever
$\phi$ is a test function and $\tilde \phi = |\Phi'|^{-d} \phi
\circ \Phi^{-1}$.  (Here $|\Phi'|^d$ is the Jacobian---since $\Phi$ is conformal, this means that $|\Phi'(z)|$ is the factor by which $\Phi$ stretches distances of pairs of points near $z$.  And $(h,\phi)$ is the value of the distribution $h$
integrated against $\phi$.) Note that if $h$ is a continuous
function (viewed as a distribution via the map $\phi \to
\int \phi(z) h(z) dz$), then $h \circ
\Phi^{-1}$ is the continuous function given by the ordinary composition of $h$ and $\Phi^{-1}$ (with this function again being interpreted as a distribution).}

When $d=1$ and $s \in (1/2, 3/2)$, we have $H\in (0,1)$ and the FGF$_s$ is (up to multiplicative constant; see \cite{fgfsurvey}) the {\bf fractional Brownian motion with Hurst parameter $H$}.  It is a Gaussian random function $h: \R \to \R$, except that instead of requiring $h(0) = 0$ (which would break the translation invariance of the process) we generally consider $h$ to be defined only modulo a global additive constant.  This means that while the quantity $h(t)$ is not a well-defined random variable when $t$ is fixed, the quantity $h(t_1) - h(t_2)$ is a well-defined random variable when $t_1$ and $t_2$ are given.  The law of this process, the fractional Brownian motion with Hurst parameter $H$, is then determined by the variance formula $$\Var \bigl(h(t_1) - h(t_2)\bigr) = |t_1 - t_2|^{2H}.$$  (Covariances of the form $\Cov\bigl(h(t_1) - h(t_2), h(t_3)- h(t_4)\bigr)$ can be derived from this.)  More general  FGF$_s$ correlation formulas, which apply when $H \not \in (0,1)$, are explained in \cite{fgfsurvey}.

When $d=1$, the LGF corresponds to $s = 1/2$ and is thus in some sense the limit of the fractional Brownian motion processes as $H$ decreases to $0$.\footnote{When $d=1$ the derivative of an FGF$_s$ process is always an FGF$_{s-1}$ process, and hence all FGF$_s$ process are obtained by starting with either a fractional Brownian motion or an LGF, and then integrating or differentiating some integer number of times. From this, it is clear that if an FGF$_s$, for $s \in (1/2,3/2]$, is defined modulo additive constant, then the distributional derivatives FGF$_{s-1}$, FGF$_{s-2}$, etc.\ are defined without an additive constant.  Thus FGF$_s$ is defined as a random distribution without an additive constant when $s \leq 1/2$.  Since the FGF$_s$, for $s \in (1/2,3/2]$, is defined modulo additive constant, the indefinite integrals FGF$_{s+1}$, FGF$_{s+2}$, etc.\ are respectively defined modulo linear polynomials, quadratic polynomials, etc.\cite{fgfsurvey}}
For each $d \geq 1$, it turns out that when $s \in (d/2, d/2+1)$ the FGF$_s$ is an a.s.\ continuous random function and has the property that its restriction to any one-dimensional line has the law of a fractional Brownian motion (up to a global additive constant)\cite{fgfsurvey}.\footnote{It is shown more generally in \cite{fgfsurvey} that if $s > 0$, $d \geq 1$, and $k\geq 1$ then the FGF$_s$ in dimension $d$ can in some sense be obtained as the restriction of a $(d+k)$ dimensional FGF$_{s+k/2}$ to a $d$ dimensional plane.}
In this range, FGF$_s$ is often called a {\em fractional Brownian field} and has been studied in a variety of contexts (see e.g.\ \cite{MR1190370, bojdecki1999fractional, zhu2002parameter}).   If $s = (d+1)/2$ then the restriction of FGF$_s$ to any one-dimensional line is a Brownian motion (up to a global additive constant), and the FGF$_s$ itself is known as {\em L\'evy's Brownian motion} \cite{Levy1954, mckean1963brownian, levy1965processus, ciesielski1975levy}.  The LGF in any dimension $d \geq 1$ can be understood as the limit of these (continuous) fractional Brownian fields obtained when $s$ decreases to $d/2$.

When $s \leq d/2$ the FGF$_s$ can only be defined as a random distribution, and not as a continuous function.  In this sense, the value $s=d/2$ corresponding to the LGF is critical: it divides FGF$_s$ fields that (like Brownian motion) can be defined as random a.s.\ continuous functions from those that (like the LGF and the GFF) can only be defined as random generalized functions \cite{fgfsurvey}.


\subsection{Basic definitions} \label{subsec::basicdefs}
We recall a few basic definitions and facts that can be found in many textbooks on Fourier analysis (see, e.g., Chapter 11 of \cite{MR1829589}).  Fix $d \geq 0$.   The {\bf Schwartz space} $\Ss$ is the space of smooth complex-valued functions on $\mathbb R^d$ whose partial derivatives all decay super-polynomially.  In other words,
$$\Ss = \{ \phi \in C^{\infty}(\mathbb R^d)\,\,\,\, : \,\,\,\,\|\phi\|_{\alpha, \beta} < \infty \,\,\, \forall \alpha, \beta \},$$ where $\alpha$ and $\beta$ are multi-indices (i.e., $d$-tuples of non-negative integers) and $$||\phi||_{\alpha, \beta} = \sup_{x \in \mathbb R^d} |x^\alpha D^\beta \phi(x) |,$$
where $x^\alpha$ and $D^\beta$ respectively mean product of multi-order $\alpha$ of  the $d$ coordinates of $x$, and partial derivative of multi-order $\beta$ with respect to those coordinates. One can define a countable family of semi-norms by $$\phi \to \sup_{x \in \mathbb R^d} |x^{\alpha} D^\beta \phi(x)|,$$ for multi-indices $\alpha$ and $\beta$.  These semi-norms are actually norms when restricted to $\Ss$, and they induce a locally convex topology on $\Ss$, w.r.t.\ which $\Ss$ is metrizable, complete, and separable \cite{pietsch1972nuclear}.  A {\bf tempered distribution} $h$ is a continuous linear map from the Schwartz space to $\mathbb C$.  For $\phi \in \Ss$, we write $(h, \phi)$ for the value of this map applied to $\phi$.  Let $\T$ denote the space of tempered distributions.  Derivatives and integrals of tempered distributions are defined by integration by parts
$$(D^\beta h, \phi) = (-1)^{|\beta|}(h, D^\beta \phi),$$ where $|\beta|$ is the sum of the indices of $\beta$, and the product of $h$ with a function like $x^{\alpha}$ can be defined by
$$(x^{\alpha} h, \phi) = (h,x^\alpha \phi).$$

At first glance, $\T$ appears to be a fairly large and unwieldy space.  However, the {\em Schwartz representation theorem} states that for any $u \in \T$, there is a finite collection $u_{\alpha \beta}: \R^d \to \C$ of bounded continuous functions, $|\alpha| + |\beta| \leq k$, for some $k < \infty$, such that $$u = \sum_{|\alpha| + |\beta| \leq k} x^{\beta} D^\alpha u_{\alpha \beta}.$$  In other words, tempered distributions are just (finite sums of) products of polynomials and derivatives of bounded continuous functions \cite{melrose2007introduction}.  Another way to say this is that $\T$ is the smallest linear space that contains all bounded continuous functions and is closed under differentiation and monomial multiplication.  (Equivalently, it is the smallest linear space that includes the bounded continuous functions and is closed under differentiation and the Fourier transform --- see the Fourier transform discussion below.)  It is often natural to equip $\T$ with the weak-$*$ topology (i.e., the weakest topology for which the maps $h \to (h, \phi)$ are continuous for each $\phi \in \Ss$).  In this topology $h_1, h_2, \ldots$ converge to $h$ if and only if $(h_1, \phi), (h_2, \phi), \ldots$ converge to $(h, \phi)$ for all $\phi \in \Ss$.

We also use $(\cdot, \cdot)$ to denote the standard $L^2(\R^d)$ inner product defined by $(f,g) = \int f(x)\bar g(x)dx$, where $\bar z$ denotes the complex conjugate of $z$.  (One can define this inner product for elements of $\Ss$ and extend it to all of $L^2(\R^d)$ by noting that $L^2(\R^d)$ is the Hilbert space completion of $\Ss$ w.r.t.\ this norm.)  Throughout this paper, we define the Fourier transform using the normalization that makes it a unitary transformation on the complex function space $L^2(\R^d)$, namely:
$$\hat f(\omega)  := (2\pi)^{-d/2} \int_{\R^d} f(x) e^{-i\omega \cdot x}dx,$$
so that
$$f(x) = (2\pi)^{-d/2} \int_{\R^d} \hat f(\omega) e^{i \omega \cdot x} d\omega.$$

The Fourier transform is a continuous one-to-one operator on Schwartz space that changes differentiation $D^\alpha$ to multiplication by $x^{\alpha}$, and vice versa.
The Fourier transform also preserves the $L^2$ norm.  Since $\Ss$ is dense in $L^2(\R^d)$, the operator can be continuously extended to a map from $L^2(\R^d)$ to $L^2(\R^d)$.  The Fourier transform defined on $\Ss$ also induces a one-to-one operator on the space $\T$ of tempered distributions: if $h$ is a tempered distribution, then the definition of $\hat h$ is fixed by the identity for $\phi\in\mathcal S$\begin{equation} \label{eqn::temperedft} (\hat h, \hat \phi) = (h, \phi).\end{equation}  Note that if $h \in L^2(\R^d)$ then the map $\phi \to (h,\phi)$ defined via the $L^2(\R^d)$ inner product can also be understood as a tempered distribution.  When $h \in L^2(\R^d)\subset \T$, the two definitions of the Fourier transform given above  (for $L^2(\R^d)$ and for $\T$) coincide.  We stress again that the Fourier transform fixes each of the three spaces $\Ss \subset L^2(\R^d) \subset \T$.

A {\bf tempered distribution modulo additive constant} can be understood in two ways: as an equivalence class of tempered distributions (with two tempered distributions considered equivalent if their difference is a constant function), or as a continuous linear functional defined only on the {\em subspace} $\Ss_0$ of Schwartz space consisting of functions $\phi$ with $\int \phi(z)dz = 0$.  It is not hard to see that these two notions are equivalent.  A continuous linear map $\phi \to (h, \phi)$ on $\Ss$ is determined by its restriction to $\Ss_0$ together with the value of $(h,\phi)$ for one fixed $\phi \in \Ss \setminus \Ss_0$ (and knowing this value is equivalent to knowing the global additive constant).  Let $\T_0$ denote the space of tempered distributions modulo additive constant.

The Fourier transform of an element of $\T_0$, defined via \eqref{eqn::temperedft} as above, is a continuous linear functional on the Fourier transform of $\Ss_0$, i.e., on the {\em subspace} $\hat{\Ss_0}$ of Schwartz space consisting of functions that vanish at zero.  Accordingly, denote by $\hat{\T_0}$ the set of continuous linear functionals on $\hat{\Ss_0}$, so that $\hat{\T_0}$ is the Fourier transform of $\T_0$.

A standard example of a random distribution is white noise: recall that the {\bf white noise} $W$ on $\R^d$ is the so-called ``standard Gaussian'' random variable associated with the $L^2$ inner product $(\cdot, \cdot)$ on $\R^d$. This means that if $f_1, f_2, \ldots$ are an orthonormal basis for $L^2(\mathbb R^d)$, then an instance of white noise can be written as $W = \sum_i \alpha_i f_i$, where the $\alpha_i$ are i.i.d.\ $N(0,1)$ (i.e., normal with mean zero, variance one).  The sum $\sum_i \alpha_i f_i$ almost surely diverges pointwise, but for each $\rho \in L^2(\R^d)$, we can write $$(W, \rho) := \sum_i (\alpha_i f_i, \rho),$$ a sum that converges almost surely.  The random variables $(W,\rho)$, for $\rho \in L^2(\R^d)$ form a Hilbert space of centered (i.e., mean-zero) Gaussian random variables with covariances given by
\begin{equation} \label{eqn::whitenoisedef} \Cov\bigl[ (W, \rho_1), (W,\rho_2) \bigr] = (\rho_1,\rho_2). \end{equation}
for all test functions $\rho_1$ and $\rho_2$ in $L^2$.  Hilbert spaces of centered Gaussian random variables (where the covariance is the inner product) are discussed in much more detail in the reference text \cite{janson1997gaussian}.

Note that, in particular, if $\rho_1$ and $\rho_2$ have disjoint support then $(W,\rho_1)$ and $(W, \rho_2)$ are independent.  Also, although there a.s.\ exist exceptional elements $\rho \in L^2(\R^d)$ for which $\sum_i (\alpha_i f_i, \rho)$ does not converge, the sum does a.s.\ converge for all elements of $\Ss$, and this allows $W$ to be a.s.\ defined as an element of $\T$.   This fact is explained (with more references and detail) in \cite{fgfsurvey}.  White noise can be defined as the unique centered Gaussian random tempered distribution for which \eqref{eqn::whitenoisedef} applies.

A {\bf complex white noise} takes the form $W_1 + i W_2$ where $W_1$ and $W_2$ are independent and each is a (real) white noise as defined above.  The Fourier transform of a complex white noise is itself a complex white noise.

Denote by $\Delta$ the Laplacian operator.  Note that $\Delta$ is a continuous operator on $\Ss$.  If $f \in \T$ then we can define $-\Delta f$ via the integration by parts formula $$(-\Delta f, \phi) := (f, -\Delta \phi),$$ for $\phi \in \mathcal S$.  Thus the Laplacian also makes sense as a map from $\T$ to itself.
 Next, observe that if $\rho = (-\Delta) f$ then $\hat \rho(\omega) = \hat f(\omega) |\omega|^2$.
That is, the operation $-\Delta$ corresponds to multiplication of the Fourier transform by the function $\omega \to |\omega|^2$, in the sense that if $\phi \in \Ss$ then $$\bigl( (-\Delta) f ,\phi \bigr) = (|\cdot|^2 \hat f, \hat \phi) := (\hat f, |\cdot|^2 \hat \phi).$$

We can now {\em define} the powers $(-\Delta)^{s/2}$ of the Laplacian, where $s \in \mathbb R$, as the operators that multiply a function's Fourier transform by $|\cdot|^{s}$.
\begin{equation} \label{eqn::lappowerdef} \bigl( (-\Delta)^{s/2} f, \phi \bigr) := (|\cdot|^{s} \hat f, \hat \phi) = (\hat f, |\cdot|^{s} \hat \phi).\end{equation}
However, we observe that multiplication by $|\cdot|^{s}$ describes a continuous function from $\Ss$ to itself only when $s/2$ is a non-negative integer: for other values of $s$, the RHS of \eqref{eqn::lappowerdef} only formally makes sense if all derivatives of $\hat \phi$ vanish at zero.  Thus, in the most straightforward sense, our definition of Laplacian powers only applies for positive integer values of $s/2$.  Another way to say this is to note that the middle expression in \eqref{eqn::lappowerdef} is not necessarily defined, because the product of a tempered distribution and a not-necessarily-smooth function like $|\cdot|^{s}$ is not necessarily a tempered distribution.

However, we can still consider on a case-by-case basis whether we can make sense of $|\cdot|^{s} \hat f$ as a tempered distribution.  If not, we can always at least make sense of the RHS of \eqref{eqn::lappowerdef} for some subspace of test functions $\phi$.   (For example, the RHS of \eqref{eqn::lappowerdef} makes sense at least whenever $\hat \phi$ and all of its derivatives vanish at the origin.)  We take \eqref{eqn::lappowerdef} to be the definition of $(-\Delta)^{s/2} f$, with the understanding that this restriction to a subspace of test functions is sometimes necessary. 

A more in depth discussion of the domain and range of the operator $(-\Delta)^{s/2}$ appears in the FGF survey \cite{fgfsurvey}, which also contains additional references to the functional analysis literature.
What will be important for us is that $(-\Delta)^{s/2} f$ is well-defined as a tempered distribution modulo a global additive constant when $f$ is a complex white noise (in which case $\hat f$ is also a complex white noise) and $s = - d/2$.  This fact is derived in \cite{fgfsurvey} as a consequence of the Bochner-Minlos theorem.  In this case, $|\cdot|^s \hat f$ is a complex white noise times the deterministic real valued function $|\cdot|^s$.  This implies that if $\hat \phi$ is a real test function in the Schwartz space then the real part of $(|\cdot|^s \hat f, \hat \phi)$ has variance given by $\int \bigl( |\omega|^s \hat \phi(\omega) \bigr)^2 d\omega = \int |\omega|^{-d} \bigl(\hat \phi(\omega)\bigr)^2 d\omega$, which is easily seen to be finite if and only if $\hat \phi(0) = 0$.  (Similar statements hold for the imaginary part.)  In this case, considering $(-\Delta)^{s/2} f$ only as a tempered distribution modulo additive constant (i.e., as an element of $\T_0$) corresponds to restricting to the subspace $\Ss_0$ of test functions $\phi$ for which $\hat \phi(0) = 0$.


\subsection{LGF definitions} \label{ss:lcgfdefs}

We will generally define the LGF to be a real-valued generalized function $h$ (so that $(h,\phi)$ is real when $\phi$ is real).  However, we remark that it is often natural to consider a complex analog of the LGF by writing $h_1 + i h_2$, where each $h_i$ is an independent real LGF.  We will give several equivalent definitions of the LGF on $\R^d$ below, each of which has a real and a complex analog.  In the definitions involving Fourier transforms, note that requiring $h$ to be real is formally equivalent to imposing the symmetry $\hat h(z) = \overline{\hat h(-z)}$.  We will explain some of the equivalences directly in the statement of Proposition \ref{prop::defequiv} below.

\begin{proposition} \label{prop::defequiv}
Each of the LGF definitions listed below describes a random element $h$ of $\T_0$.  The corresponding probability measures on $\T_0$ are equivalent (up to multiplying $h$ by a deterministic constant).
\end{proposition}

\begin{enumerate}
\item \label{def::correlations} A random $h \in \T_0$ (i.e., a random distribution modulo global additive constant) whose law is the centered Gaussian determined by the covariance formula \begin{equation} \label{eqn::coveq} \Cov\bigl[ (h, \phi_1), (h, \phi_2) \bigr] = \int_{\R^d \times \R^d} -\log|y-z|  \phi_1(y) \phi_2(z)dydz,\end{equation} for all test functions $\phi_1$ and $\phi_2$ in $\Ss_0$.
\item \label{def::Gaussian_in_Fourier_Hilbert_space} A ``standard Gaussian'' in a particular Hilbert space that we view as a subspace of $\T_0$.  To construct that Hilbert space, first note that each element of $\Ss$ can be interpreted as an element of $\T_0$ (by considering its equivalence class modulo global additive constant).  Then consider the Hilbert space closure of $\Ss \subset \T_0$ w.r.t.\ the inner product \begin{equation} \label{eqn::fourierlcgfdef} (f,g)_d := \bigl(|\cdot|^d  \hat f ,\hat g \bigr) = \bigl( \hat f ,|\cdot|^d \hat g \bigr) = \int_{\R^d} |z|^d \hat f(z)\overline{\hat g(z)} dz.\end{equation}  
The statement that $h$ is a ``standard Gaussian'' means that we can write \begin{equation}\label{eqn::convergencelcgfdef} h = \sum_i \alpha_i f_i \end{equation} where $\alpha_i$ are i.i.d.\ centered normal random variables and the $f_i$ are an orthonormal basis for the Hilbert space described above.  The sum converges a.s.\ within the space $\T_0$.
\item \label{def::Gaussian_in_Laplacian_L2_Hilbert_space} A standard Gaussian in a particular Hilbert space: namely, the Hilbert space closure of $\Ss \subset \T_0$ w.r.t.\ the inner product \begin{equation} \label{eqn::deltadef} (f,g)_d := \bigl( (-\Delta)^{d/2} f, g\bigr) =\bigl(f, (-\Delta)^{d/2} g\bigr)=\bigl((-\Delta)^{d/4} f, (-\Delta)^{d/4} g\bigr),\end{equation}  where $(\cdot, \cdot)$ denotes the standard $L^2$ inner product and $\Delta$ is the Laplacian.  This definition is formally equivalent to Definition \ref{def::Gaussian_in_Fourier_Hilbert_space} once one posits the identity $$\widehat{(-\Delta)^a f} = |\cdot|^{2a} \hat f$$ and the fact that $f \to \hat f$ is an isometry of $L^2(\R^d)$.  This definition is especially natural when $d$ is a multiple of $4$, since in this case $(-\Delta)^{d/4}$ is a local operator on $\mathcal S$, and $(f,f)_d$ is simply the square of the $L^2$ norm of $(-\Delta)^{d/4} f$.
    Note that $(-\Delta)^{d/2}$ is still a local operator if $d$ is merely a multiple of $2$, which gives us an easy way to define the middle expressions of \eqref{eqn::deltadef}.  For odd values of $d$, one has to define $(-\Delta)^{d/4}$ via Fourier transforms (which amounts to returning to Definition \ref{def::Gaussian_in_Fourier_Hilbert_space}).
\item \label{def::Gaussian_in_Laplacian_Dirichlet_Hilbert_space} A standard Gaussian in a particular Hilbert space: namely, the Hilbert space closure of $\Ss \subset \T_0$ w.r.t.\ the inner product $$(f,g)_d := \bigl( (-\Delta)^{(d-2)/2} f, g\bigr)_\nabla =\bigl(f, (-\Delta)^{(d-2)/2} g\bigr)_\nabla =\bigl((-\Delta)^{(d-2)/4} f, (-\Delta)^{(d-2)/4} g\bigr)_\nabla,$$  where $(\cdot, \cdot)_\nabla$ denotes the standard Dirichlet inner product $(f,g)_\nabla := \int \nabla f(x) \cdot \nabla g(x) dx$.  This is equivalent to Definition \ref{def::Gaussian_in_Laplacian_L2_Hilbert_space} by the identity $$(f,g)_\nabla = ( -\Delta f, g) = (f, -\Delta g).$$  Nonetheless, this definition is especially natural when $d-2$ is a multiple of $4$, since in this case $(-\Delta)^{(d-2)/4}$ is a local operator and $(f,f)_d$ is simply the Dirichlet energy of $(-\Delta)^{(d-2)/4}f$.
\item \label{def::Laplacian_noise} $(-\Delta)^{-d/4} W$ where $W$ is a real white noise. (Since $W$ is by definition the standard Gaussian in the Hilbert space $L^2(\R^d)$, where the inner product is $(f,g)$, this definition is equivalent to Definition \ref{def::Gaussian_in_Laplacian_L2_Hilbert_space}.)
\item \label{def::Fourier_noise} The real part of the Fourier transform of $|z|^{-d/2} W(z)$, where $W$ is complex white noise on $\R^d$.  (This is formally equivalent to Definition \ref{def::Gaussian_in_Fourier_Hilbert_space}, since $|z|^{-d/2} W(z)$ is the standard Gaussian in the Hilbert space with inner product given by the RHS of \eqref{eqn::fourierlcgfdef}.)
\end{enumerate}

\begin{proof}
The Hilbert spaces described are all equivalent by definition.  The fact that the sums converge in $\mathcal T_0$ is shown in \cite{fgfsurvey}. \qed
\end{proof}
\vspace{.1in}

\begin{proposition}
If $\phi$ is a conformal automorphism of $\R^d$, then the law of $h \circ \phi$ agrees with the law of $h$.
\end{proposition}
\begin{proof} Consider the inversion map $\phi(z) = z/|z|^2$ and suppose that $\mu$ and $\nu$ are finite signed measures on $\R^d$, each with total mass zero.  Then either direct calculation or  geometrical reasoning gives $|z||w||z \frac{1}{|z|^2}-w \frac{1}{|w|^2}|=|z \frac{|w|}{|z|}-w \frac{|z|}{|w|}|=|z-w|$ and hence $$\int \log|z-w| d\mu(z) d\nu(w) = \int \log| \phi(z)  - \phi(w) | d\mu(z) d\nu(w).$$  By Definition \ref{def::correlations} of Proposition \ref{prop::defequiv}, this implies the invariance of the law of the LGF under the inversion map $\phi$.  A similar calculation implies invariance when $\phi$ is a translation or a dilation.  The proposition now follows from the well known fact (Liouville theorem) that for $d>2$ all conformal automorphisms of $\R^d$  are compositions of inversions, translations and dilations.
\qed
\end{proof} 
\vspace{.1in} 

Next, we would like to say that, in some sense, applying the operator ``multiplication by $|\cdot|^{-d/2}$'' (which fails to be defined on the entire space of tempered distributions) to a function's Fourier transform corresponds to applying the operator ``convolution with $|\cdot|^{-d/2}$'' (which also fails to be defined on the entire space of tempered distributions) to the function itself.  The first definition in the proposition below is inspired by this idea.

\begin{proposition} \label{prop::defequiv2}
The following are equivalent to the definitions given in Proposition \ref{prop::defequiv}.
\end{proposition}
\begin{enumerate}
\item \label{def::noise_convolution} A formal convolution of a real white noise $W$ and the function $|z|^{-d/2}$.
        Formally, this means that $h(a) = \int W(z) |a-z|^{-d/2} dz$, but this integral is almost surely undefined for any given $a$ (which makes sense, since the $h$ we seek to define this way is a distribution, not a function, and cannot be defined pointwise).  However, one can define $h_\varepsilon$ to be the convolution of $W$ with $$\psi_\varepsilon(z) := |z|^{-d/2} 1_{\varepsilon < |z | <1/\varepsilon}.$$ This object can be defined pointwise, and we can construct $h$ by taking the limit in law as $\varepsilon \to 0$.
\item \label{def::cone} The formal integral \begin{equation} \label{eqn::coneformula} h(a) = \int_{C(a)} W\bigl((x,y)\bigr)dxdy,\end{equation} where $C(a) := \{(x,y): x \in \R^d
    \,\,\,\mathrm{and}\,\,\, 0 \leq y \leq |x-a|^{-d} \}$ and $W$ is a real white noise on $\mathbb R^d \times [0, \infty)$.  The $h$ defined this way is a distribution, not a function, and can be defined precisely by first replacing $C(a)$ with the subset $C_\varepsilon(a):= C(a) \cap \{(x,y): \varepsilon < y < \varepsilon^{-1} \}$ (so that the the approximation $h_\varepsilon$ thus defined is a random continuous function), and we can construct $h$ by taking the limit in law as $\varepsilon \to 0$.
\end{enumerate}

\begin{proof}
We will sketch the proof here.  In each of the two cases above, we can compute the variance of $(h_\varepsilon, \rho)$ explicitly (assuming $\rho$ is Schwartz with mean zero) and check that as $\varepsilon \to 0$ it converges to the variance of $(h,\rho)$ when $h$ an instance of the LGF.  Since the $h_\varepsilon$ are all Gaussian, it follows from this that the $h_\varepsilon$ converge in law to the LGF as $\varepsilon \to 0$.  For any fixed $\rho$, the convergence also holds in $L^2$, which implies almost sure convergence of $(h_\varepsilon, \rho)$ to a limit.  A similar argument shows that for any countable dense collection of $\rho_i$, the restriction of $(h_\varepsilon, \cdot)$ converges almost surely to a limit with the law of the LGF restricted to these test functions.
\qed
\end{proof}

Definition \ref{def::cone} in Proposition \ref{prop::defequiv2} involves integrating white noise over an extra dimension of space: intuitively one has a different white noise for each $y$ (a convolution of white noise with the indicator function of a ball) and $h$ is the integral of these processes over $y \in [0,\infty)$.  This construction will in fact turn out to be closely related to another construction of the LGF as a continuum analog of the additive Mandelbrot cascade, which we will describe in Section \ref{sec::cascadecomparison}.  This construction, explained in more detail in Section \ref{sec::cascadecomparison}, involves averaging a continuum of rescaled instances of a stationary bounded-variance Gaussian random function on $\R^d$.
Section \ref{sec::cutoffs} will describe a variety of ways to obtain the LGF as an integral of this sort, following the method of Kahane.

One can find in \cite{fgfsurvey} more detail on how to enlarge the space of test functions beyond the Schwartz space.  We cannot make sense of $(h, \rho)$ when $\rho$ is a $\delta$-function (which would correspond to $h$ being defined pointwise) but we can sometimes make sense of $(h,\rho)$ when $\rho$ represents a measure supported on a set of dimension strictly between $0$ and $d$, so that, for example, $(h, \rho)$ represents the mean value of $h$ on a line-segment.  We will see that \eqref{eqn::coveq} still holds for the random variables $(h, \rho)$ defined this way.  This allows one to make precise the following so-called {\bf restriction property} of the LGF (which is a special case of a more general result explained in \cite{fgfsurvey}):
\begin{proposition} \label{prop::restriction}
If $h$ is an LGF on $\R^d$, and a lower-dimensional subspace $\Lambda$ of $\R^d$ is fixed, then the restriction of $h$ to $\Lambda$ is an LGF on $\Lambda$.
\end{proposition}Proposition \ref{prop::restriction} is actually immediate from \eqref{eqn::coveq} once we know that $(h,\rho)$ is well-defined for test functions $\rho$ supported on the lower dimensional subspace.  The reader familiar with the 2D GFF may find this property intriguing.  It implies that no matter what $d$ is, the restriction of $h$ to any fixed two-dimensional slice of $\R^d$ is just an ordinary 2D GFF.  The higher dimensional LGF can thus be understood as a way to couple together all of these Gaussian free fields (one for each two-dimensional plane).  All of the structure that exists for the 2D GFF can be found on a two-dimensional slice of an LGF in $\R^d$.  For example, as mentioned earlier, it was shown in \cite{ContourLine} that even though the 2D GFF is a distribution, not a continuous function, it is possible to define continuous zero-height ``contour lines'' of the GFF, which are forms of the Schramm-Loewner evolution (SLE).  Although we are not able to do this, it seems natural to try to rigorously construct a level surface of the 3D LGF (as in Figure \ref{fig::levelsurfaces}) by joining together the contour lines defined on each plane in a family of parallel planar slices of $\R^3$.

Sections \ref{sec::markov} and \ref{sec::cutoffs} will establish Markov properties of $h$, decompositions of $h$ into independent sums, joint laws for the averages of $h$ on circles and spheres, and various schemes for approximating $h$.


\subsection{Acknowledgements}

Scott Sheffield learned much about this subject from personal conversations with Linan Chen and Dmitry Jakobson, whose recent work \cite{2012arXiv1210.8051C} considers random distributions similar to the LGF on $\R^4$, along with variants defined on compact and possibly curved manifolds.   They also consider analogs of the Liouville quantum gravity measures for these objects and establish a general expected box counting version of the KPZ formula for these measures. Vincent Vargas would like to thank Vincent Millot for useful discussions on fractional Laplacians.  Scott Sheffield thanks the astronomer John M.\ Kovac for conversations about early universe cosmology and for the corresponding references.  Kovac is the leader of the BICEP2 team, which has been collecting and analyzing cosmic microwave background radiation data at the South Pole.  Their recent detection of so-called $B$-modes in this data has been widely discussed and analyzed this year \cite{ade2014detection}, and interpretation efforts are ongoing.  Sheffield also thanks Asad Lodhia, Xin Sun, and Samuel Watson, his co-authors on the FGF survey \cite{fgfsurvey}, for a broad range of useful insights.

\section{Comparison to cascades} \label{sec::cascadecomparison}
If one types ``log-correlated Gaussian field'' into an online search engine, one finds that there has a been a good deal of research on random fields that have {\em approximately} logarithmic correlations (as opposed to the {\em exactly} logarithmic correlations enjoyed by the LGF).\footnote{In contrast to our specific usage here, the phrase ``log-correlated Gaussian field'' is sometimes used more broadly to describe approximately logarithmic fields.}  In practice, many of the tools that are useful for studying the LGF can be applied equally well to Gaussian random functions with approximately logarithmic correlations.

\begin{figure}[ht!]
\begin{center}
\subfigure[$Y_1$]{\includegraphics[width=0.45\textwidth]{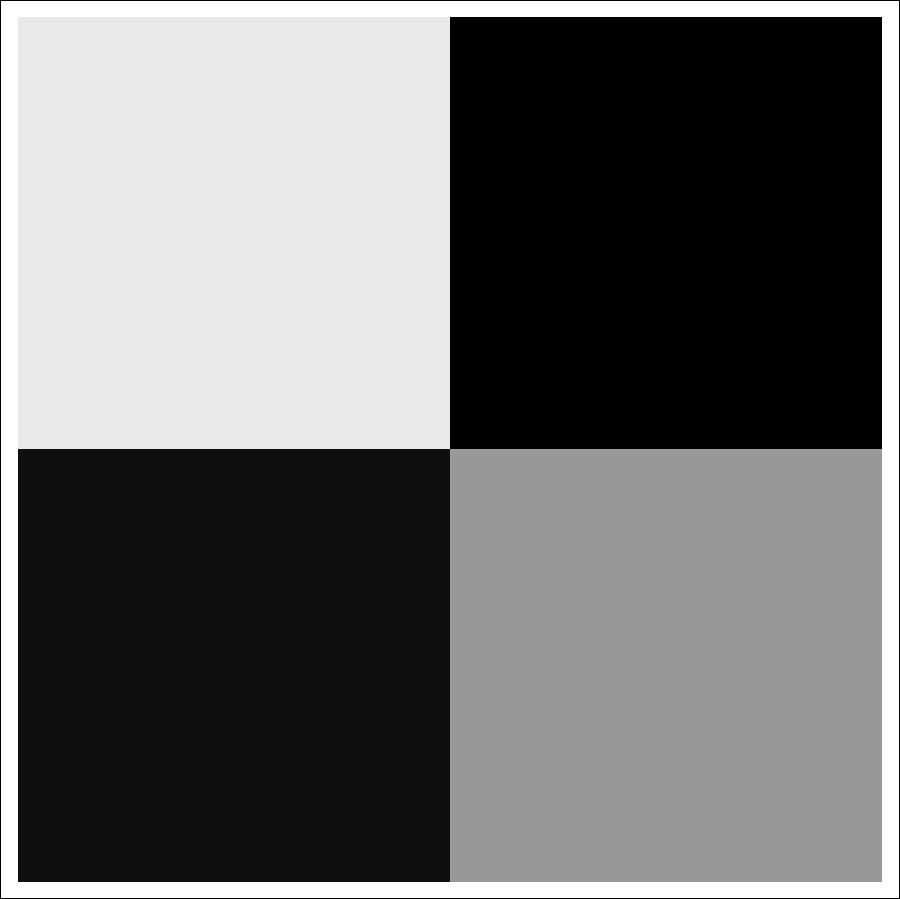}}
\hspace{0.02\textwidth}
\subfigure[$Y_2$]{\includegraphics[width=0.45\textwidth]{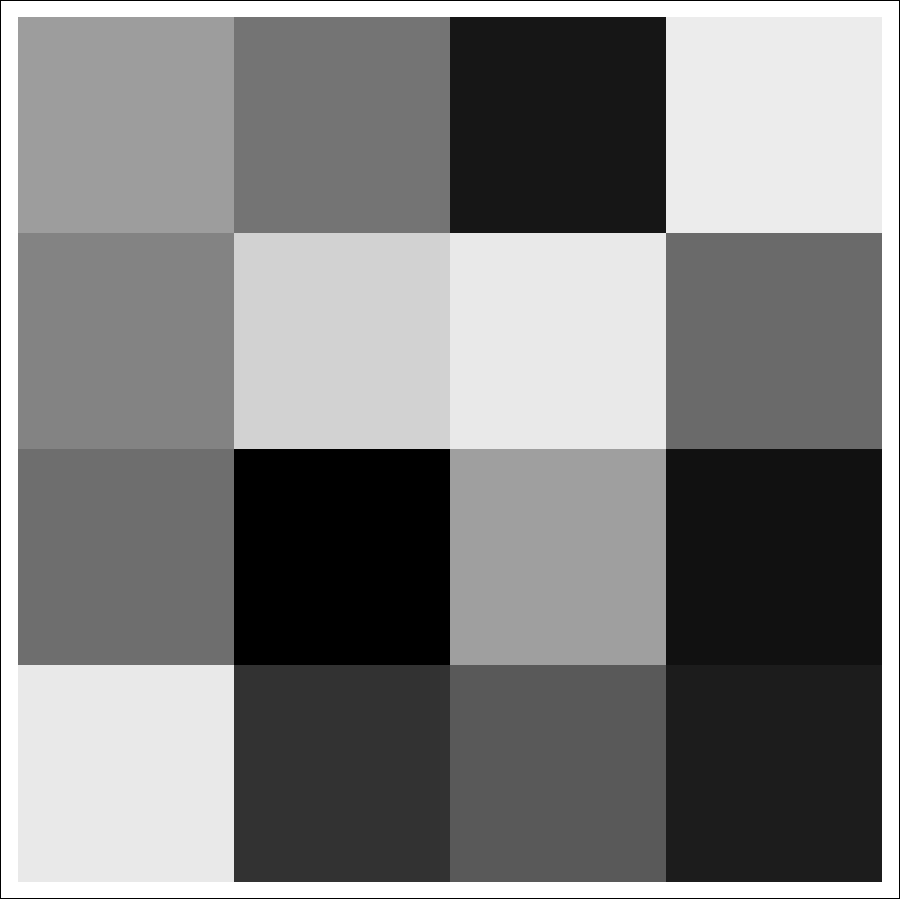}}
\subfigure[$Y_3$]{\includegraphics[width=0.45\textwidth]{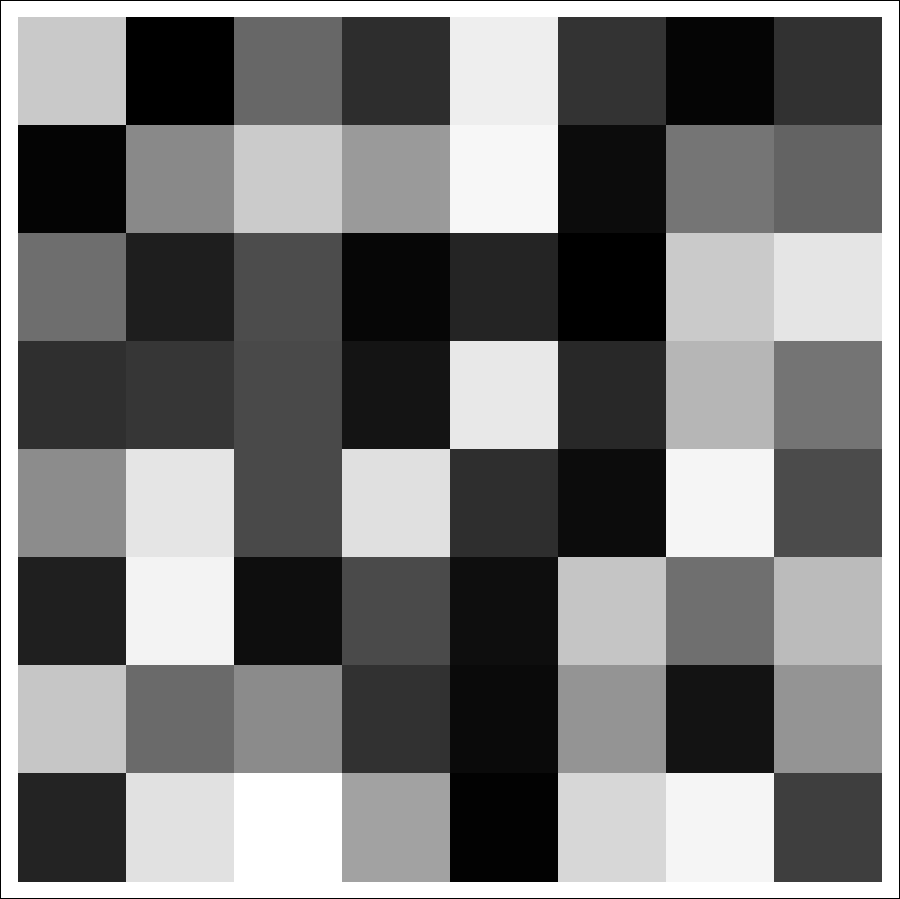}}
\hspace{0.02\textwidth}
\subfigure[$Y_4$]{\includegraphics[width=0.45\textwidth]{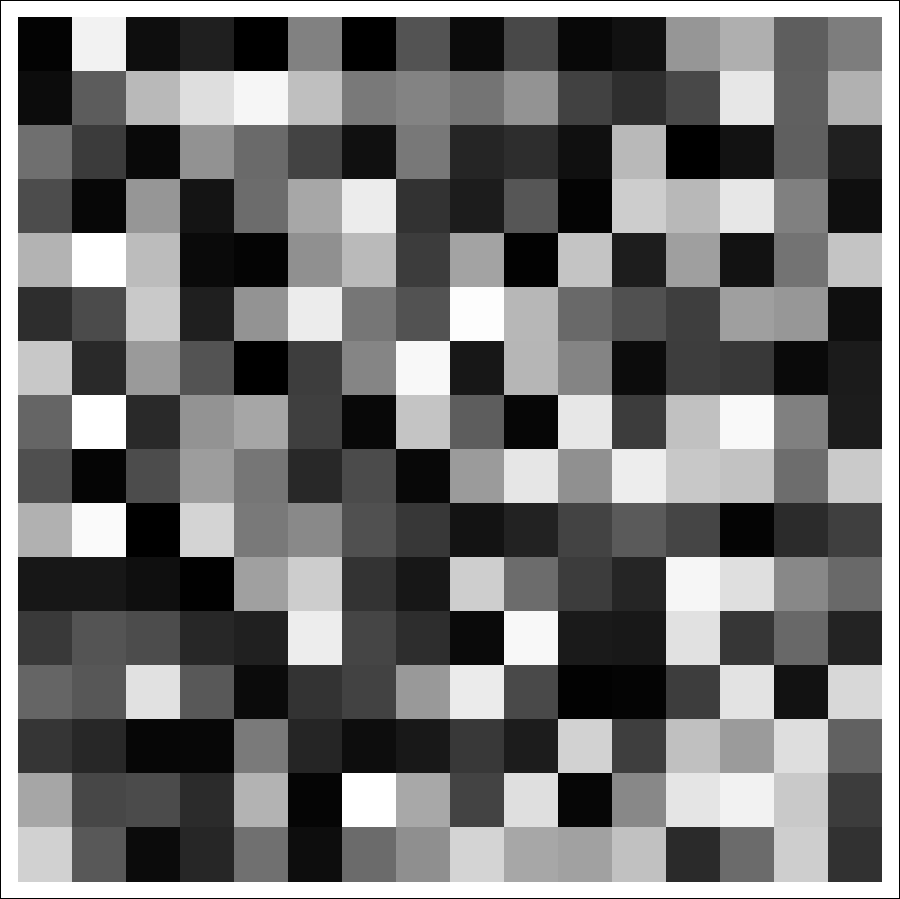}}
\end{center}
\caption{\label{fig::Ylevels} Possible instances of the independent random functions $Y_1$,  $Y_2$,  $Y_3$, and $Y_4$, restricted to the unit square $[0,1]^2$ (with function values indicated by grey scale).  The sum $\sum_{k=0}^\infty Y_k$ is a well-defined random distribution on $\R^2$.  The sum $\sum_{k=-\infty}^\infty Y_k$ is well-defined as a random distribution, modulo global additive constant, on $[0,\infty)^2$.}
\end{figure}

To motivate the definition of the LGF and build some intuition, we will first describe one of these closely related families of random generalized functions: the so-called {\bf additive cascades} constructed from branching random walks.  In their simplest form, additive cascades are defined as follows.  Let $Y_0$ be a random function from $\R^d$ to $\R$ that assigns an independent mean-zero, unit-variance Gaussian random variable to each unit cube of $\mathbb Z^d$ --- i.e., each cube obtained by translating $[0,1]^d$ by an element of $\mathbb Z^d$.  (The value assigned to the measure-zero boundaries between cubes does not matter.)  For each integer $k$, let $Y_k$ be an independent random function on $\R^d$ such that \begin{equation} \label{eqn::y0yk} Y_k(\cdot) \,\, =_{\textrm{LAW}} \,\, Y_0 (2^k \cdot).\end{equation}
Thus, for each $k$ the function $Y_k$ is a.s.\ piecewise constant, assigning a different random value to each cube obtained by translating $[0,2^{-k}]^d$ by an element of $2^{-k}\mathbb Z^d$, as illustrated in Figure~\ref{fig::Ylevels}.  We mention without proof the following facts.  (These are not hard to verify, but it would be a distraction to give detailed proofs here.)
\begin{enumerate}
\item The infinite sum $\sum_{k=0}^\infty Y_k$ a.s.\ converges in the space of distributions to a limiting random distribution (generalized function) $Y^+$ whose law is invariant under translations of $\R^d$ by integer vectors (elements of $\mathbb Z^d$).  In other words, there is a.s.\ a distribution $Y^+$ such that for each smooth, compactly supported test function $\phi$, we have $(Y^+, \phi) = \sum_{k=0}^\infty (Y_k, \phi)$.
\item The infinite sum \begin{equation}\label{eqn::yinfty} Y = Y_0 + \sum_{k=1}^\infty (Y_k + Y_{-k})\end{equation} also converges to a limit in the space of distributions understood {\em modulo additive constant} (which amounts to limiting the space of test functions $\phi$ to those whose global mean is zero) if we restrict our attention to $[0,\infty)^d \subset \R^d$.  This follows from the convergence of $\sum_{k=0}^\infty Y_k$ and the fact that if $\phi$ is compactly supported in $[0,\infty)^d$, then the function $Y_{-k}$ is a.s.\ constant on the support of $\phi$ for all sufficiently large $k$.  However, the discontinuity of $Z_n := Y_0 + \sum_{k=1}^n (Y_k + Y_{-k})$ at the origin becomes arbitrarily large as $n \to \infty$, so the limit does not converge on all of $\R^d$.
\item Suppose $x_1, x_2$ are distinct points in $[0,\infty)^d$ with irrational coordinates.  Let $L(x_1,x_2)$ be the smallest $k$ such that $x_1$ and $x_2$ lie in different cubes of $2^{-k} \mathbb Z^d$.  Observe that $Y_k(x_1) = Y_k(x_2)$ a.s.\ when $k < L(x_1,x_2)$ and $Y_k(x_1)$ and $Y_k(x_2)$ are independent when $k \geq L(x_1,x_2)$.  This implies that \begin{equation} \label{eqn::covxn} \Cov\bigl( Z_n(x_1), Z_n(x_2) \bigr) = n + L(x_1, x_2) \end{equation} whenever $n$ is larger than $|L(x_1,x_2)|$.  When we compute covariances of differences (and $n > |L(x_1,x_2)|$) the $n$ terms on the right side of \eqref{eqn::covxn} cancel out:
    $$\Cov \bigl[ Z_n(x_1) -Z_n(x_2) , Z_n(x_3) - Z_n(x_4) \bigr] = L(x_1, x_3) - L(x_1, x_4) - L(x_2, x_3) + L(x_2,x_4).$$  Taking the $n \to \infty$ limit, we obtain that for any smooth mean-zero test functions $\phi_1$ and $\phi_2$ compactly supported on $[0,\infty)^d$,
        \begin{equation}\label{eqn::covy} \Cov\bigl[ (Y, \phi_1), (Y, \phi_2) \bigr] = \int_{[0,\infty)^d \times [0,\infty)^d} L(y,z)  \phi_1(y) \phi_2(z)dydz.\end{equation}
\end{enumerate}

The LGF is a continuum version of the random distribution $Y$ defined above, satisfying \eqref{eqn::coveq} in place of \eqref{eqn::covy}, which amounts to replacing $L(x,y)$ by $-\log|x-y|$.  (Note that the function $L(x,y)$ is in some sense a discrete analog of the function $-\log_2|x-y|$.)
However, unlike $Y$, the LGF is defined on all of $\R^d$ and it has a law that is invariant under {\em arbitrary} translations, rotations, and dilations (and not only dilations by powers of $2$).  Nonetheless, we will see in Section \ref{sec::cutoffs} that the LGF can still be defined via \eqref{eqn::y0yk} and \eqref{eqn::yinfty} provided we replace $Y_0$ with an appropriate stationary bounded-variance Gaussian random function on $\R^d$ (indeed there are many choices for $Y_0$ that suffice).
This $Y_0$ can take the form of a white noise convolved with a bump function, so that $Y_0(x_1)$ and $Y_0(x_2)$ are independent when $|x_1-x_2|$ is large enough.  Alternatively, we can let $Y_k$ denote the portion of the integral in \eqref{eqn::coneformula} corresponding to $2^{dk} < y < 2^{d(k+1)}$.  (It is easy to see that \eqref{eqn::y0yk} holds for the $Y_k$ defined this way.)

Thus, much of the intuition suggested by the cascade model (where each $k$ represents a ``scale'' and there is an independent constant-order-variance noise at each scale) is valid for the LGF as well.
\begin{remark}
A more general class of random distributions on $[0,\infty)^d$ (modulo additive constant) can be constructed by replacing $\sum_{k \in \mathbb Z} Y_k$ with $\sum_{k \in \mathbb Z} \alpha^k Y_k$.  The latter are invariant under transformations that dilate the domain $[0,\infty)^d$ by a factor of $2$ and the range $\R$ by a factor of $\alpha$.  Just as $\sum_{k \in \mathbb Z} Y_k$ is analogous to the LGF, the more general family $\sum_{k \in \mathbb Z} \alpha^k Y_k$ is analogous to the general family of FGFs.
\end{remark}

\begin{remark}
If in the construction of $Y$ above we made $Y_0$ constant on integer translations of $[-1/2,1/2]^d$, instead of integer translations of $[0,1]^d$, then the corresponding random distribution $Y$ could be defined on all of $\R^d$, instead of just on $[0,\infty)^d$, since the support of any compactly supported test function lies in $2^k [-1/2,1/2]^d$ when $k$ is large enough.
\end{remark}

\section{Orthogonal decompositions and Markov properties} \label{sec::markov}

\subsection{Conditional expectation on a bounded domain} \label{sec::eigenfunctions}
Consider a bounded subdomain $D \subset \R^d$.  Let $\harm_D$ be the $(f,g)_{LGF}$ Hilbert-space closure of the space of smooth functions $f$ for which $(-\Delta)^{d/2} f = 0$ on $D$.  Let $\supp_D$ be the Hilbert-space closure of the space of smooth, compactly supported functions on $D$.  The following is a relatively simple observation (see \cite{fgfsurvey} for a proof of a generalization of this statement that applies to general $d$ and $s$; an analogous result for the GFF appears in \cite{MR2322706}):

\begin{proposition} The spaces $\harm_D$ and $\supp_D$ are orthogonal subspaces, and indeed they form an orthogonal decomposition of the entire LGF Hilbert space described by \eqref{eqn::hspaceinnerproductdef}.
\end{proposition}

The projection of an instance $h$ of the LGF onto $\harm_D$ gives the conditional expectation of $h$ in $D$ {\em given} the values of $h$ outside of $D$.   Similarly, we can define the zero-boundary LGF on $D$ to be the projection of $h$ onto $\supp_D$ (a construction explained more generally in \cite{fgfsurvey}).  When $d$ is even, the operator $(-\Delta)^{d/2}$ is simply a positive integer power of the ordinary Laplacian, and the functions in $\harm_D$ satisfy $\Delta^{d/2} f = 0$.  Thus, to understand conditional expectations of $h$ on $D$ (given the observed field outside of $D$) one needs to understand what functions in $\harm_D$ are like.

\subsection{Radially symmetric functions}

In light of the previous section, it will be useful to us to better understand functions $g$ that satisfy $(-\Delta)^{d/2} g = 0$.  We begin with spherically symmetric functions.  The following is a straightforward calculation of the Laplacian of powers of the $d$-dimensional norm function $x \to |x|$:

\begin{equation} g(x) = |x|^k  \,\, \implies \,\, \Delta g(x) = k (k + d -2) |x|^{k-2}, \end{equation}

and

\begin{equation} g(x) = \log |x| \,\, \implies \Delta g(x) = (d-2) |x|^{-2}. \end{equation}

Note that $\Delta g(x) = 0$ away from the origin if $d=2$ and $g(x) = \log|x|$; and similarly, $\Delta g(x) = 0$ away from the origin if $d \not = 2$ and either $g(x) = |x|^{2-d}$ (in which case $g(x)$ is a multiple of Green's function) or $g(x) = |x|^{0}$.  This implies the following:

\begin{proposition} \label{eqn::evenharmonic}
Suppose $d \geq 2$ is even.  We have $(-\Delta)^{d/2} g(x) = 0$ for $x\in \mathbb R\setminus \{0\}$ whenever $g(x) = \log|x|$.  If $g(x) = |x|^k$ then we have $(-\Delta)^{d/2}g(x) = 0$ for $x\in \mathbb R\setminus \{0\}$ if and only if $k$ is even and $2-d \leq k \leq d-2$.
\end{proposition}

Readers familiar with the Gaussian free field in dimension $d=2$ will recall that the radially symmetric harmonic functions, namely the linear span of the function $1$ and $\log |\cdot|$, play an important role.  Within this linear span, we can find a function that takes any specified value on the circle $\partial B_r(0)$ and any other specified value on $\partial B_s(0)$, where $0 < r < s$.  The function can be interpreted as the expectation of a free field on the annulus $\{z : r < |z| < s \}$ with these boundary conditions.  Proposition \ref{eqn::evenharmonic} suggests that in a general even dimension $d$, the analogous space is spanned by $d$ spherically harmonic functions.  Within the linear span of these functions, we can find a function such that it and its first $d/2-1$ radial derivatives take specified values on $\partial B_r(0)$ and $\partial B_s(0)$.  We interpret this function as the expectation of the LGF with the specified values and derivative values on the annulus boundary.







\subsection{Other spherical harmonics}
We follow here the spherical coordinate decomposition given for the GFF in \cite{jerison2014}. We write the Laplacian in spherical coordinates as 		 \begin{equation} \label{e.sphericallaplace} \Delta = r^{1-d} \frac{\partial}{\partial r}r^{d-1} \frac{\partial}{\partial r} + r^{-2} \Delta_{S^{d-1}}.\end{equation}
A polynomial $\psi \in \R[x_1,\ldots,x_d]$ is called \emph{harmonic} if $\Delta \psi$ is the zero polynomial.   (It is called \emph{bi-harmonic}, \emph{tri-harmonic}, etc. if $\Delta^2 \psi = 0$, $\Delta^3 \psi = 0$, etc.)
Suppose that $\psi$ is harmonic and homogeneous of degree~$k$.  Letting $f = \psi |_{S^{d-1}}$, we have $\psi(ru) = f(u) r^{k}$ for all $u \in S^{d-1}$ and $r \geq 0$.  Setting \eqref{e.sphericallaplace} to zero at $r=1$ yields
	\[ \Delta_{S^{d-1}} f = -k(k + d - 2) f, \]
i.e., $f$ is an eigenfunction of $\Delta_{S^{d-1}}$ with eigenvalue $-k(k+d-2)$.
Note that the expression $-k(k + d-2)$ is unchanged when
the nonnegative integer $k$ is replaced with the negative
integer $k' : = -(d-2)-k$.  Thus $f(u)r^{k'}$ is also
harmonic on $\R^d \setminus \{0\}$.  We can also use \eqref{e.sphericallaplace} to precisely derive the bi-harmonic, tri-harmonic, etc.\ functions that are given by $f(u) g(r)$ in spherical coordinates where $f$ is spherical harmonic.

We mention a few basic results about spherical harmonics that appear, e.g., in \cite{stein1971introduction}.  Assume $d \geq 2$.  Let $A_k$ be the set of homogenous degree $k$ harmonic polynomials on $\R^d$ and let $H_k$ be the space of functions on the unit sphere $S^{d-1} \subset \R^d$ obtained by restriction from $A_k$.  Then the dimension of $H_k$ is given by $$\binom{d+k-1}{d-1} - \binom{d+k-3}{d-1},$$
which is finite.  An important property is that the spaces $H_k$ are pairwise orthogonal (for the standard inner product of $L^2(S^{d-1})$) and the sum is dense in $L^2(S^{d-1})$.

Now we will discuss how to describe the LGF in spherical coordinates, which will give us a natural way to describe a higher dimensional Gaussian field with a countable collection of one-dimensional Gaussian fields.  (A more general version of this story appears in \cite{fgfsurvey}.)  For each $k$, let $\{ \psi_{k,j} \}$ be an orthonormal basis of $H_k$, where $1 \leq j \leq \mathrm{dim} \: H_k$.  Then let $H_{k,j}$ denote the space of functions on all of $\mathbb R^d$ that have the form $$\psi(ur) = \psi_{k,j}(u) f(r),$$ whenever $u$ is in the unit sphere and $r \in [0,\infty)$, with $f:[0,\infty) \to \mathbb R$ being a continuous and sufficiently smooth function.  More precisely, let $H_{k,j}$ be the Hilbert space closure, with respect to the inner product $(\cdot, \cdot)_d$, of the set of smooth functions of this type.

Note that the LGF is a standard Gaussian on the Hilbert space spanned by the orthogonal subspaces $H_{k,j}$.  Thus, we can write an instance $h$ of the LGF as a sum $$h = \sum_{k=0}^\infty \sum_{j=1}^{\mathrm{dim} H_k} h_{k,j},$$
where the $h_{k,j}$ are independent standard Gaussians on the space $H_{k,j}$.

Note that by the definition of the spaces $H_{k,j}$, we can write: $$h_{k,j}(ru) = \tilde h_{k,j}(r) \psi_{k,j}(u),$$ for $u \in S^{d-1}$ and $r \in [0,\infty)$, where $\tilde h_{k,j}$ is a one-dimensional Gaussian field from $[0,\infty)$ to $\mathbb R$.

In order to describe $\tilde h_{k,j}$, it is useful to recall the spherical coordinate Laplacian \eqref{e.sphericallaplace}. Indeed, we have

\begin{align}
\Delta \psi(ur) &= r^{1-d} \frac{\partial}{\partial r} \bigl( r^{d-1}  f'(r) \bigr) \psi_{k,j}(u)  -k(k+d-2) r^{-2} \psi(ur)  \\
&= \bigl( r^{-1} (d-1) f'(r) + f''(r) -k(k+d-2) r^{-2} f(r) \bigr) \psi_{k,j}(u).  \\
\end{align}

This gives us a definition of the Laplacian restricted to the space $H_{k,j}$, i.e.
\begin{equation*}
\mathcal{L}_{k,d}(r,f)(r):=  r^{-1} (d-1) f'(r) + f''(r) -k(k+d-2) r^{-2} f(r) .
\end{equation*}
 This restriction can now be viewed as a one-dimensional object: i.e, an operator that maps a function on $[0,\infty)$ to another function on $[0,\infty)$.  Therefore,  $\tilde h_{k,j}$  is the standard Gaussian on the Hilbert space defined by the scalar product
\begin{equation*}
\int_0^\infty r^{d-1} (-\mathcal{L}_{k,d})^{d/2}(r,f)(r)   g(r) dr.
\end{equation*}
When $d$ is a multiple of four, we can write this as 
\begin{equation*}
\int_0^\infty r^{d-1} (-\mathcal{L}_{k,d})^{d/4}(r,f)(r)  (-\mathcal{L}_{k,d})^{d/4}(r,g)(r) dr.
\end{equation*}
When $d$ is two more than a multiple of $4$, we can use the identity $$(f,g)_d = \bigl( (-\Delta)^{(d-2)/4} f, (-\Delta)^{(d-2)/4}g\bigr)_\nabla.$$  In each case, we find that the corresponding norm obtained on functions $f:\R_+ \to \R$ is an integral over $r$ of a quantity depending on $r, f(r), f'(r), f''(r), \ldots, f^{(d/2+1)}(r)$.

It turns out that the $h_{k,j}$ are continuous functions when $d\geq 2$ and are $d/2-1$ times differentiable for general even $d>2$ \cite{fgfsurvey}.  From this and the above discussion it follows that if we are given the values and first $d/2-1$ derivatives of $\tilde h_{k,j}$ at a fixed value $s$, then conditionally the restriction of $\tilde h_{k,j}$ to $[0,s)$ is independent of the restriction of $\tilde h_{k,j}$ to $(s,\infty]$.

By combining the results for each of the $H_{k,j}$ spaces, we obtain the entire conditional law of the LGF inside of a sphere of radius $s$, given the values of the LGF outside.  We find in particular that this law only depends on the values and first $d/2-1$ derivatives of each of the  $\tilde h_{k,j}$ at $s$.  This is the {\em Markov property} mentioned in the introduction.  Since the LGF is conformally invariant, these observations also determine the conditional law of the LGF on one side of a $d-1$ dimensional hyperplane given its values on the other side.

Finally, we remark that one could in principle also write $r = e^t$ for $t \in \mathbb R$, and write the Laplacian in terms of the $t$ parameter.  In this way, one would obtain a map from functions on all of $\mathbb R$ to functions on all of $\mathbb R$.  The (small) advantage to this approach is that the conformal symmetry with respect to the change $r \to 1/r$ then reduces to a symmetry with respect to the change $t \to -t$.  In $t$ coordinates, one has both translation and reflection invariance of the Laplacian operator.

\section{Cutoffs and approximations} \label{sec::cutoffs}

\subsection{Kahane decomposition} \label{sec::kahane}
We discuss different kinds of ``cutoffs'' which are essentially ways to write the LGF (or analogs on unbounded or bounded domains) as a sum of two independent pieces, one of which looks like the LGF on small scales (but looks approximately constant on large scales) and one of which looks like the LGF on large scales (and is approximately constant on small scales).

The most straightforward way to do this is to recall from Proposition \ref{prop::defequiv} that when $h$ is a complex LGF, we can write its Fourier transform $\hat h(z) = |z|^{-d/2}W$, where $W$ is white noise.  One can write $\hat h$ as a sum of two terms as follows: $$\hat h(z) =|z|^{-d/2} 1_{|z| \leq 1} W(z) + |z|^{-d/2} 1_{|z| > 1} W(z),$$ and applying the inverse Fourier transform to these pieces gives a way to write $h$ as a sum of two terms. Then, one can decompose further and write $h = \sum_{j= -\infty}^\infty h_j$ where
$$\hat h_j(z) = |z|^{-d/2} 1_{2^{j} \leq z < 2^{j+1}} W(z).$$
Since $h_j$ and $h_{j+1}$ agree in law (up to a scaling of space by a constant factor) this construction is rather analogous to the additive cascade model in Section \ref{sec::cascadecomparison}.

\subsubsection{Whole plane approximations}\label{wpa}
We will now discuss a decomposition scheme (in the spirit of the multiplicative chaos constructions of Kahane \cite{MR829798,allez}) that allows us to decompose the LGF into a continuum of independent pieces, each corresponding to a different scale, in a translation and scale invariant way.  See \cite{2012arXiv1206.1671D,MR3215583} for more details and references.

We consider a family of  centered stationary Gaussian processes $((X_{t}(x))_{x \in \R^d})_{t\geq 0}$ where, for each $t\geq 0$, the process $(X_{t}(x))_{x \in \R^d}$ has  covariance given by:
\begin{equation} \label{corrX}
K_t(x)=\E[X_{t}(0)X_{t}(x)]= \int_1^{e^t}\frac{k(ux)}{u}\,du,
\end{equation}
for some  covariance kernel $k$ satisfying $k(0)=1$, of class $C^1$ and vanishing outside a compact set.  It is natural to let $k$ be the covariance kernel describing the convolution of white noise on $\R^d$ with a compactly supported test function. A convenient construction  of $X$ is the following. Write $k$ as the Fourier transform of a smooth nonnegative function $g\in L^2(\R^d)$ and define
$$f(x)=(2\pi)^{-d/4}\widehat{\sqrt{g}}$$ in order to get $k=f*f$. Then we consider a white noise $W$ on $\R^d\times \R_+$ and we define
$$X_t(x)=\int_{\R^d\times [1,e^t]}f(y-ux)u^{-1/2}\,W(dy,du).$$
 It can be checked that
 $$\E[X_t(x)X_{t'}(x')]=\int_1^{e^{\min(t,t')}}\frac{k(u(x-x'))}{u}\,du,$$ in such a way that \eqref{corrX} is valid and  the process $(X_t(x)-X_s(x))_{x\in\R^d}$ is independent of the processes $\big((X_u(x))_{x\in\R^d}\big)_{u\leq s}$ for all $s<t$. Put in other words, the mapping $t\mapsto X_t(\cdot)$ has independent increments. This allows to stack independent innovations so as to recover, as $t\to\infty$, a Gaussian random distribution $X$ with covariance kernel given by
 $$\E[X(x)X(x')]=\int_1^{\infty}\frac{k(u(x-x'))}{u}\,du.$$

The important thing is that in this Kahane approach, when one integrates $\log u$ from $0$ to $\infty$, one obtains (modulo additive constant) the whole plane LGF.  The object one obtains by just integrating between $0$ and $t$ (as we do to obtain $X_t$) can be understood as the expectation of the LGF given the information in this range.

\subsubsection{Kahane's cutoff on bounded domains}

Kahane originally considered centered random Gaussian distributions $X$ the covariance kernels of which can be rewritten as
$$K(x,y)=\sum_{n\geq 1}k_n(x,y)$$ for some sequence $(k_n)_n$ of continuous covariance kernels. Thus there exists a sequence $(Y^n)_n$ of independent centered Gaussian processes such that
$$\E[Y^n(x)Y^n(y)]=k_n(x,y).$$
A sequence of cutoffs/approximations of $X$ is then given by
$$\forall x\in D\subset \mathbb R^d,\quad X^n(x)=\sum_{k=1}^nY^k(x).$$
From now on, we stick to the
notations of Subsection \ref{sec::eigenfunctions}. As suggested in \cite{MR2819163} in the case of the two-dimensional GFF, a natural choice of cutoffs/approximations of the LGF can be achieved by  choosing
$$k_n(x,y)=(-\lambda_n)^{-\frac{d}{2}}e_n(x)e_n(y),$$
 corresponding to $Y^n(x)=(-\lambda_n)^{-\frac{d}{4}} \beta_n e_n(x)$, where $\lambda_n$ and $e_n$ denote the eigenvalues and eigenfunctions of the Dirichlet Laplacian on $D$, and $\beta_n$ i.i.d. centered normal random variables. The important point here is that the approximating sequence $(X^n)_n$ can be almost surely defined as a function of the whole LGF distribution $X$ in terms of projections onto an orthonormal basis of the Dirichlet Laplacian.

In dimension $2$, a construction of the GFF with Dirichlet boundary conditions, close to that in Subsection \ref{wpa}, has been suggested in \cite{rhodes-2008}. Consider a bounded open domain $D$ of $\R^2$. The Green function of the Laplacian can then be rewritten as
\begin{equation*}
G_D (x,y)=   \int_{0}^{\infty}p_D(t,x,y)dt,
 \end{equation*}
where $p_D$  is the (sub-Markovian) semi-group of a Brownian motion $B$ killed upon touching the boundary of $D$, namely for a Borel set $A\subset D$,
$$\int_A p_D(t,x,y)\,dy=P^x(B_{t} \in A, \; T_D > t), $$ with $T_D=\text{inf} \{t \geq 0, \; B_t\not \in D \}$.  
Note that, for each $t>0$, $p_D(t,x,y)$ is a continuous symmetric and positive definite kernel on $D$. Therefore, by considering a white noise $W$ distributed on $D\times \R_+$, we define
 $$X_t(x)=\int_{D\times [e^{-2t},\infty]}p_D(\frac{s}{2},x,y)W(dy,ds).$$
One can check that
$$\E[X_t(x)X_t(x')]=\int_{e^{-2t}}^\infty p_D(s,x,x')\,ds,$$
 and  that the process $(X_t(x)-X_s(x))_{x\in D}$ is independent of the processes $\big((X_u(x))_{x\in D}\big)_{u\leq s}$ for all $s<t$, allowing the possibility of dealing with independent innovations. For finite $t$, $X_t$ is a cutoff approximation of the GFF that is obtained by taking $t=\infty$.

\subsection{LGF in dimension 1: financial volatility and cutoffs}
We briefly recall a cutoff-based construction of the LGF that appeared in work on the volatility of financial markets involving one of the current authors \cite{DRV}.  It is proposed in \cite{BaMu}, based on some empirical data, that the log volatility of many assets (e.g., stocks, currencies, indices) can be modeled by a Gaussian random distribution $h_T$ living in $\T$ and with the following covariance structure (for test functions $\phi_1,\phi_2$):
\begin{equation}\label{eq:MRM}
\Cov\bigl[ (h_T, \phi_1), (h_T, \phi_2) \bigr] = \int_{\R \times \R} \log_{+} \frac{T}{|y-z|}  \phi_1(y) \phi_2(z)dydz,
\end{equation}
where $\log_{+}(x)=\max(\log(x) , 0)$ and $T$ is the so-called correlation length. When calibrated on real data, one sometimes finds a $T$ which is larger than the calibration window (see \cite{BaMu}).
If $T$ is larger than the window of time being considered, then it makes sense to replace $T$ by $\infty$ in \eqref{eq:MRM} and consider the log volatility as an element of $\T_0$. This was the framework considered in \cite{DRV} to forecast volatility.  This gives another construction of the LGF in dimension 1, namely as the limit when $T \to \infty$ of $h_T$ in the space $\T_0$.

\bibliographystyle{halpha}
\bibliography{lcgf}
\end{document}